\documentclass[11pt]{article}
\usepackage[T1]{fontenc}
\usepackage[utf8]{inputenc}
\usepackage{lmodern,amsmath,amssymb,amsthm,mathtools,booktabs}
\usepackage[margin=1.1in]{geometry}
\usepackage[colorlinks=true,linkcolor=blue,citecolor=blue,urlcolor=blue]{hyperref}
\usepackage{microtype}
\hypersetup{pdftitle={Coordinate descents, monodromy, and finite normalization of marked-root maps},
  pdfauthor={Jos\'e A. R. Fonollosa},
  pdfkeywords={Keller maps, Jacobian conjecture, coordinate descent, monodromy, finite normalization}}
\newtheorem{theorem}{Theorem}[section]
\newtheorem{proposition}[theorem]{Proposition}
\newtheorem{lemma}[theorem]{Lemma}
\newtheorem{corollary}[theorem]{Corollary}
\theoremstyle{definition}

\theoremstyle{remark}
\newtheorem{remark}[theorem]{Remark}
\newcommand{\A}{\mathbb A}
\newcommand{\C}{\mathbb C}
\newcommand{\Z}{\mathbb Z}
\newcommand{\Gm}{\mathbb G_m}
\newcommand{\PP}{\mathbb P}

\DeclareMathOperator{\Gal}{Gal}
\DeclareMathOperator{\Pic}{Pic}
\DeclareMathOperator{\sgn}{sgn}
\DeclareMathOperator{\disc}{disc}

\title{Coordinate descents, monodromy, and finite normalization\\
of marked-root maps}
\author{Jos\'e A. R. Fonollosa\thanks{Universitat Polit\`ecnica de Catalunya, Barcelona, Spain.
  \texttt{jose.fonollosa@upc.edu}. ORCID:
  \href{https://orcid.org/0000-0001-9513-7939}{0000-0001-9513-7939}.}}
\date{September 2026}
\begin{document}
\maketitle
\begin{abstract}
We study the dimension-indexed marked-root polynomial maps introduced by
Harish. For every $n\ge4$, their lower coefficient block extends to a
polynomial coordinate system, giving three-dimensional Keller maps of
geometric degree $n(n-2)$ for every choice of the fixed coefficients.
We determine the generic monodromy of the original maps and all successive
coordinate descents: it is the cyclic wreath product for odd $n$, and an
explicit subgroup of index two for even $n$. Its order is
$(n-2)^n n!/\gcd(2,n-2)$.
We construct the smooth finite normalization and identify its complement
of the affine source as $n-2$ disjoint affine divisors. One is ramified;
the others record unramified loss of inverse sheets. The normalization
has Picard group $\Z^{n-2}$ and the homotopy type of a wedge of $n-2$
two-spheres. Its rational deck group is cyclic, while the polynomial
deck group of the source map is trivial. By base change, the
normalization describes the preimage of any subvariety of the target;
when the target polynomial factors, a component with only constant units
is an open subset of a Kummer cover of the simple roots of one factor.
Finally, the monodromy separates these maps, for $n\ge4$, from single
weighted lifts under stable polynomial left--right equivalence.


(secondary).
\end{abstract}
\setcounter{tocdepth}{1}
{\small\tableofcontents}
\clearpage
\section{Introduction and principal results}

All varieties and function fields are over $\C$.

\subsection{Keller maps and the marked-root construction}

A \emph{Keller map} is a polynomial map $F:\A^n\to\A^n$ whose Jacobian
determinant is a nonzero constant. Such a map is a local analytic
isomorphism at every point, and the Jacobian conjecture, going back to
Keller \cite{Keller1939CremonaTransformationen}, asserts that it is then
a global polynomial automorphism; see
\cite{BassConnellWright1982DegreeReduction,VanDenEssen2000PolynomialAutomorphisms}
for background. Two numerical invariants recur throughout:
\begin{itemize}
\item the \emph{geometric degree} of $F$, which is the degree of its
function-field extension $\C(\A^n)/F^*\C(\A^n)$, or equivalently the
number of points in a generic fibre;
\item its \emph{generic monodromy}, the permutation group of the sheets
over a small open subset of the target.
\end{itemize}
The geometric degree should not be confused with the degrees of the
coordinate polynomials. A Keller map is an automorphism exactly when its
geometric degree is one.

Harish constructed, for every $n\ge3$, Keller maps $K_n:\A^n\to\A^n$ of
geometric degree $n(n-2)$ \cite{Harish2026MarkedRoot}. Since
$n(n-2)>1$, each $K_n$ is a counterexample to the Jacobian conjecture in
dimension $n$. The conjecture in dimension two remains open, and this is
what makes the lower-dimensional restrictions of $K_n$ studied here
interesting. The idea of the construction is easiest to state on the
target side.
\begin{itemize}
\item A target point is a polynomial $P(T)$ of degree at most $n$ whose
coefficient of $T$ equals one; its remaining $n$ coefficients are the
target coordinates.
\item A generic source point over $P$ is a pair $(t,a)$ consisting of a
root $t$ of $P$ and a \emph{marking} $a$ with $a^{n-2}=P'(t)$.
\end{itemize}
For a generic $P$ there are $n$ roots and $n-2$ markings for each, which
explains the degree $n(n-2)$. The source coordinates are recovered from
$(t,a)$ by explicit formulas, recalled in Section~\ref{sec:root}. The
construction, the constant Jacobian, the complete fibre formula, the
image and the nonproperness locus of $K_n$ (the set of target points
near which the map fails to be proper \cite{Jelonek1993NonProperSet})
are due to Harish; we recall the construction with full derivations to
fix conventions.

For the original cubic example, a coordinate-free recognition of the
affine source and a precise tangency criterion appear in
\cite[Appendix A]{VanDobbenDeBruyn2026ProjectiveBundles}. The present
paper concerns the whole dimension-indexed family.

\subsection{Questions addressed}

The description by roots and markings raises three natural questions.
\begin{enumerate}
\item \emph{Descent.} Fixing some target coefficients restricts $K_n$ to
a subfamily. Is each such restriction again a Keller map on an affine
space, so that one obtains Keller maps in dimension three?
\item \emph{Monodromy.} The root-and-marking description suggests the
wreath product $\mu_{n-2}\wr S_n$ as monodromy group, but the markings
may satisfy relations. What is the actual group, and is it preserved by
specialization?
\item \emph{Completion.} Rational inverse formulas describe sheets that
escape to infinity only after one knows which finite completion the
source sits in. What is the finite normalization, and which divisors
are missing from the source?
\end{enumerate}

\subsection{Notation}

Put $m=n-2$. We name the target coefficients by
\begin{equation}\label{eq:pencil}
 P(T)=qT^n+rT^{n-1}+\sum_{j=2}^{n-2}\lambda_jT^j+T-b .
\end{equation}
For $n=3$ the sum is empty, and $(r,q,b)=(A_2,C_3,A_0)$.
For $n\ge4$, Section~\ref{sec:root} writes the same coefficients as
$C_n=q$, $C_{n-1}=r$, $C_j=\lambda_j$ for $3\le j\le n-2$,
$A_2=\lambda_2$ and $A_0=b$.

The coefficients $\lambda_2,\ldots,\lambda_{n-2}$ are called the
\emph{lower coefficients}. A \emph{successive coordinate descent} fixes
a prefix $\lambda_2,\ldots,\lambda_k$ to complex values and keeps the
remaining lower coefficients and $r,q,b$ free. The empty prefix is the
original map, up to a permutation of its outputs. We write $d$ for the
dimension of the resulting target, so $3\le d\le n$. The case $d=3$,
with all lower coefficients fixed, gives maps to $\A^3_{r,q,b}$.

\subsection{Principal results}

\begin{theorem}[Polynomial coordinate descent]\label{thm:descent-main}
For every $n\ge4$, the lower coefficient functions $(A_2,C_3,\ldots,C_{n-2})$
of $K_n$ are part of a polynomial coordinate system on its source. The
coordinate changes and their inverses are constructed recursively, and
they fix the source coordinate $x$. Consequently every fixed tuple
$\lambda\in\C^{n-3}$ gives a Keller map
\[
 E_{n,\lambda}:\A^3\longrightarrow\A^3_{r,q,b}
\]
of geometric degree $nm$, in source coordinates $(x,s_{n-1},s_n)$
constructed in Section~\ref{sec:descent}. The same holds at every intermediate step of a
successive descent, for every value of the fixed coefficients.
\end{theorem}

In other words, after a polynomial change of source coordinates, $K_n$
becomes a family $(\lambda,\xi)\mapsto(\lambda,E_{n,\lambda}(\xi))$
parametrized by the lower coefficients. This is not a product with a
single map: $E_{n,\lambda}$ genuinely depends on $\lambda$. The proof in
Section~\ref{sec:descent} rests on a simple completion lemma and on the
behaviour of the coefficients along the divisor $x=0$.

The three-dimensional endpoints need not be new equivalence classes of
maps. In particular, $E_{4,0}$ is polynomially left--right equivalent to
the degree-eight power-weighted map of \cite{Annie2026PowerWeightedLifts};
Section~\ref{sec:power-overlap} gives the coordinate identity. Our claim
concerns the uniform coordinate completion of the whole family, not the
novelty of each resulting map.

\begin{theorem}[Generic monodromy]\label{thm:monodromy-main}
For $K_n$ and every successive coordinate descent, the monodromy group
acting on the $nm$ sheets is
\begin{equation}\label{eq:group}
 G_n=\begin{cases}
 \mu_m^n\rtimes S_n,&n\text{ odd},\\[2pt]
 \bigl\{(\zeta_1,\ldots,\zeta_n;\sigma)\in\mu_m^n\rtimes S_n:
 (\prod_i\zeta_i)^{m/2}=\sgn(\sigma)\bigr\},&n\text{ even}.
 \end{cases}
\end{equation}
In particular $|G_n|=m^n n!/\gcd(2,m)$.
\end{theorem}

Here $\sigma$ permutes the roots and $\zeta_i$ rotates the marking of
the $i$-th root. For odd $n$ every combination occurs. For even $n$ the
markings satisfy one relation, coming from the classical identity
between the product of the values $P'(t_i)$ and the discriminant; this
cuts the wreath product down to an index-two subgroup. We give that
subgroup as a subset of the wreath product and do not claim that it
splits over $S_n$. Yang reported numerical monodromy of order $192$ for
$K_4$ \cite{Yang2026ThreeCharts}. Our proof obtains this order
algebraically and gives the result for all $n$ and all fixed parameters,
by combining valuations at a colliding pair of roots with a Morse
specialization (Section~\ref{sec:monodromy}).

\begin{theorem}[Finite normalization and boundary]\label{thm:normal-main}
Let $F:\A^d\to Y=\A^d$ be $K_n$ or a successive coordinate descent, and
let $Z$ be the normalization of $Y$ in the function field $\C(\A^d)$ via
$F$. Then $Z$ is smooth and finite flat of degree $nm$ over $Y$. The
source embeds as an open subset $U\simeq\A^d$ of $Z$, and $Z\setminus U$
is a disjoint union of $m$ copies of $\A^{d-1}$:
\begin{enumerate}
\item one \emph{critical} divisor, with generic ramification index $2m$
and residue degree one, which normalizes the binary discriminant;
\item $m-1$ unramified divisors, each mapping isomorphically onto the
hyperplane $\{b=0\}$ of polynomials that vanish at $T=0$.
\end{enumerate}
Moreover
\[
 \Pic(Z)\simeq\Z^m,\qquad
 Z(\C)\simeq\bigvee_{j=1}^m S^2.
\]
The rational deck group over $Y$ is $\mu_m$; the polynomial deck
automorphism group of $F$ is trivial.
\end{theorem}

Geometrically, a source point disappears in two ways. Over the
discriminant two roots collide, and the markings of the colliding pair
tend to zero; these sheets form the critical divisor. Over $b=0$ the
root $t=0$ can carry only the marking $a=1$; the other $m-1$ markings
correspond to points of $Z$ at which the source coordinate $y$ has a
pole. Finite normalization is an established tool for Keller maps, see
\cite{BorisovGabberVasiuKellerDegrees}, and its use in comparing related
constructions is developed by van Rijn
\cite{VanRijn2026MarkedRootFramework,VanRijn2026NormalizationFunctoriality}.
Here we construct the normalization explicitly. Affine-chart gluing has
classical precedents for Danielewski--Fieseler surfaces
\cite{Dubouloz2004DanielewskiFieseler}, but we do not identify our
completion with a member of that class. Triviality of the polynomial deck
group holds for every complex affine-space Keller map; the new
information is the rational deck group and its action on the source
open.

\begin{corollary}[Stable distinction]\label{cor:stable-main}
For $n\ge4$, the original map and every successive coordinate descent
are stably polynomially left--right inequivalent to any single ordinary
weighted lift of the same geometric degree. Here an \emph{ordinary
weighted lift} is a map to $\A^3_{A,B,C}$ whose source function field
is generated over $\C(A,B,C)$ by one element $w$ with minimal polynomial
$\mathcal H(w)-BCw+cAC^2$, where $\mathcal H\in\C[w]$ has degree $nm$ and
$c\ne0$.
\end{corollary}

The reason is that $G_n$ is imprimitive on its $nm$ sheets (the $m$
markings of each root form a block), whereas ordinary weighted lifts have
full symmetric, hence primitive, monodromy. The symmetric monodromy of
ordinary weighted lifts is a prior result
\cite{MikhailSzh2026WeightedLiftGalois}, and stable comparison through
primitive and imprimitive actions has explicit precedents
\cite{VanRijn2026ImprimitiveFactorization}. The corollary applies those
tools to the group computed here. It does not distinguish our maps from
all power-weighted lifts or from compositions of weighted lifts.

Finally, Section~\ref{sec:sectors} base-changes the normalization along
an arbitrary morphism $\Gamma\to Y$. This describes exactly the preimage
of any subvariety of the target (Proposition~\ref{prop:base-change}),
and, for a factorization of the target polynomial over $\Gamma$, the
components carried by one factor (Proposition~\ref{prop:sector}). These
facts are the natural input for an analysis of plane sections, which is
not carried out here.

\subsection{Organization}

Section~\ref{sec:root} recalls the construction and derives its basic
identities. Section~\ref{sec:descent} proves
Theorem~\ref{thm:descent-main}, Section~\ref{sec:monodromy}
Theorem~\ref{thm:monodromy-main}, and Section~\ref{sec:normalization}
Theorem~\ref{thm:normal-main}. Section~\ref{sec:sectors} treats
preimages of subvarieties, Section~\ref{sec:comparison} proves
Corollary~\ref{cor:stable-main} and records the overlap with
power-weighted lifts, and Section~\ref{sec:verification} describes the
computational checks that accompany the proofs.

\section{The marked-root construction}\label{sec:root}

This section recalls Harish's construction and fibre theorem
\cite{Harish2026MarkedRoot}. We derive every identity used later, so
that the remaining sections can be read without the original source.

\subsection{Definition}

On $\A^n$ use coordinates $(x,y,z_3,\ldots,z_n)$ and put $p=1+xy$.
Define a polynomial in an auxiliary variable $T$ by
\begin{align}
 \Phi_n(T)={}&(1-yT)^{n-1}(pT-x)\notag\\
 &+(pT-x)^2\left[ny\left(yT+\frac12\right)
       +z_3\left(pT+\frac x2\right)
       +\sum_{j=4}^n z_jT^{j-2}\right].\label{eq:compiler}
\end{align}
It has degree at most $n$ in $T$, and we name its coefficients by
\[
 \Phi_n(T)=\sum_{j=3}^n C_jT^j+A_2T^2+T-A_0,
 \qquad K_n=(C_3,\ldots,C_n,A_2,A_0):\A^n\to\A^n.
\]
Lemma~\ref{lem:shape}(1) below shows that the coefficient of $T$ is
indeed one, so $K_n$ records all the other coefficients. For $n=3$
the sums over $j\ge4$ are empty. When an index is convenient we also
write $C_2=A_2$.

The formula is designed so that the source point knows one root of its
own image polynomial. The first line of~\eqref{eq:compiler} vanishes
at $T=x/p$ to first order, the second to second order. Hence $x/p$ is
a root of $\Phi_n$, and the derivative at that root comes from the
first line alone.

\subsection{Shape of the coefficients}

\begin{lemma}\label{lem:shape}
\leavevmode
\begin{enumerate}
\item The coefficient of $T$ in $\Phi_n$ is one.
\item For $4\le j\le n$, the variable $z_j$ enters $\Phi_n$ only
through
\[
 z_jT^{j-2}(pT-x)^2=z_j\bigl(p^2T^j-2px\,T^{j-1}+x^2T^{j-2}\bigr),
\]
and $z_3$ enters only through
\[
 z_3(pT-x)^2\Bigl(pT+\frac x2\Bigr)
 =z_3\Bigl(p^3T^3-\frac32p^2x\,T^2+\frac12x^3\Bigr).
\]
Consequently, with $z_{n+1}=z_{n+2}=0$,
\begin{equation}\label{eq:coefficient-shape}
 C_k=p^{\varepsilon_k}z_k-2px\,z_{k+1}+x^2z_{k+2}
      +(\text{a polynomial in }x,y)\qquad(3\le k\le n),
\end{equation}
where $\varepsilon_3=3$ and $\varepsilon_k=2$ for $k\ge4$. Moreover
\begin{equation}\label{eq:low-coefficients}
 A_2=x\Bigl(-\frac32p^2z_3+xz_4\Bigr)+g_n(x,y),\qquad
 A_0=x-\frac n2x^2y-\frac12x^3z_3,
\end{equation}
where $g_n(x,y)=-\frac m2y-\frac{n(n+1)}2xy^2-\frac{3n}{2}x^2y^3$
and, for $n=3$, the term $x^2z_4$ is absent.
\item On the divisor $x=0$, where $p=1$,
\begin{equation}\label{eq:boundary-coefficients}
 A_2=-\frac m2 y,\quad
 C_3=z_3+\left(\binom{n-1}{2}+n\right)y^2,\quad
 C_j=z_j+(-1)^{j-1}\binom{n-1}{j-1}y^{j-1}\quad(j\ge4).
\end{equation}
\end{enumerate}
\end{lemma}

\begin{proof}
(1) The first line of~\eqref{eq:compiler} contributes
$p+(n-1)xy$ to the coefficient of $T$. In the second line,
$(pT-x)^2=x^2-2pxT+p^2T^2$ and the bracket is
$\frac n2y+\frac x2z_3+(ny^2+pz_3)T+O(T^2)$. Their product contributes
\[
 x^2(ny^2+pz_3)-2px\Bigl(\frac n2y+\frac x2z_3\Bigr)
 =nx^2y^2-npxy=-nxy,
\]
using $p-xy=1$. The total is $p+(n-1)xy-nxy=p-xy=1$.

(2) The displayed expansions are immediate. They show that $z_j$
contributes $p^2$ to $C_j$, $-2px$ to $C_{j-1}$ and $x^2$ to $C_{j-2}$,
and that $z_3$ contributes $p^3$ to $C_3$, $-\frac32p^2x$ to $A_2$ and
nothing to the coefficient of $T$. All other terms of $\Phi_n$ involve
only $x$ and $y$. This gives~\eqref{eq:coefficient-shape} and the
$z$-dependence in~\eqref{eq:low-coefficients}. The $z$-free parts of
$A_2$ and $A_0$ are read off from the coefficients of $T^2$ and $T^0$
in~\eqref{eq:compiler} with $z=0$.

(3) At $x=0$ we have $p=1$ and
$\Phi_n=(1-yT)^{n-1}T+T^2\bigl[ny(yT+\frac12)+z_3T+\sum_{j\ge4}z_jT^{j-2}\bigr]$.
The coefficient of $T^2$ is $-(n-1)y+\frac n2y=-\frac m2y$, and the
higher coefficients are read off in the same way.
\end{proof}

By~\eqref{eq:coefficient-shape}, the matrix
$\partial(C_3,\ldots,C_n)/\partial(z_3,\ldots,z_n)$ at fixed $(x,y)$ is
upper triangular with diagonal $(p^3,p^2,\ldots,p^2)$. Its determinant
is $p^{2n-3}$. Where $p\ne0$, the high coefficients therefore determine
$z_3,\ldots,z_n$ uniquely from $(x,y)$, by back substitution starting
from $z_n$.

\subsection{The marked-root chart}

On the dense open set $xp\ne0$ put
\begin{equation}\label{eq:marked-chart}
 t=\frac{x}{p},\qquad a=\frac1p,\qquad\text{so that}\qquad
 x=\frac ta,\qquad y=\frac{1-a}{t}.
\end{equation}
This is an isomorphism from $\{xp\ne0\}\subset\A^2_{x,y}$ onto
$\{ta\ne0\}\subset\A^2_{t,a}$. Since $pt-x=0$, both lines
of~\eqref{eq:compiler} vanish at $T=t$. The second line vanishes to
second order, so only the first contributes to the derivative. With
$1-yt=(p-xy)/p=a$ this gives
\begin{equation}\label{eq:root-derivative}
 \Phi_n(t)=0,\qquad \Phi_n'(t)=(1-yt)^{n-1}p=a^{n-1}\cdot a^{-1}=a^m.
\end{equation}
So a source point with $xp\ne0$ determines a nonzero root $t$ of its
image polynomial together with an $m$-th root $a$ of the derivative
there. We call $a$ the \emph{marking} of the root.

\subsection{Determinant and fibres}

\begin{proposition}[Harish]\label{prop:harish}
The Jacobian determinant of $K_n$ is $m/2$. For a target polynomial $P$
with coefficient of $T$ equal to one, let $s(P)$ be the number of
simple roots of its degree-$n$ homogenization
$P_h(X_1,X_0)=X_0^nP(X_1/X_0)$ on $\PP^1$. Then
\begin{equation}\label{eq:fibre-count}
 \#K_n^{-1}(P)=m\,s(P)-(m-1)\mathbf1_{P(0)=0}.
\end{equation}
In particular the geometric degree of $K_n$ is $nm$.
\end{proposition}

The fibre formula has a simple reading. Every simple projective root
of $P$ carries $m$ source points, one for each marking, except the root
$t=0$, which carries only the marking $a=1$. Multiple roots carry no
source points.

\begin{proof}
\emph{The determinant.} Work on the dense open set $xp\ne0$ and factor
$K_n$ through the intermediate coordinates $(t,a,C_3,\ldots,C_n)$.
The first step $(x,y,z)\mapsto(t,a,C)$ has block lower triangular
Jacobian, since $t$ and $a$ do not depend on $z$, and
\[
 \det\frac{\partial(t,a)}{\partial(x,y)}
 =\det\begin{pmatrix}1/p^2&-x^2/p^2\\-y/p^2&-x/p^2\end{pmatrix}
 =-\frac{x(1+xy)}{p^4}=-ta^2 .
\]
Together with the triangular block, its determinant is $-ta^2p^{2n-3}$.
In the second step the high coefficients are kept, and $A_2,A_0$ are
determined by $P(t)=0$ and $P'(t)=a^m$:
\[
 A_2=\frac{a^m-1-\sum_{j=3}^n jC_jt^{j-1}}{2t},\qquad
 A_0=\sum_{j=3}^n C_jt^j+A_2t^2+t.
\]
Differentiating the second formula and using the first gives
$\partial_tA_0=a^m+t^2\partial_tA_2$ and $\partial_aA_0=t^2\partial_aA_2$.
Subtracting $t^2$ times the first row from the second,
\[
 \det\frac{\partial(A_2,A_0)}{\partial(t,a)}
 =-a^m\,\partial_aA_2=-\frac m{2t}a^{2m-1}.
\]
Moving the two rows $(A_2,A_0)$ past the $n-2$ rows $C_j$ introduces
the sign $(-1)^{2(n-2)}=1$. Multiplying,
\[
 \det JK_n=\Bigl(-\frac m{2t}a^{2m-1}\Bigr)\bigl(-ta^2\bigr)p^{2n-3}
 =\frac m2(ap)^{2n-3}=\frac m2 .
\]
Both sides are polynomials, so equality on a dense open set proves it
everywhere.

\emph{The fibres.} The source is the disjoint union of three strata:
$xp\ne0$, $x=0$ (where $p=1$), and $p=0$ (where $x\ne0$). A point of
the source determines the projective root $[x:p]$ of its image, and
the three strata correspond to nonzero finite roots, the root $t=0$,
and the root at infinity.

\emph{Nonzero finite roots.} Let $t\ne0$ be a simple root of $P$. Then
$P'(t)\ne0$, so there are exactly $m$ nonzero choices of $a$
in~\eqref{eq:root-derivative}. For each choice,~\eqref{eq:marked-chart}
gives $(x,y)$ with $p=1/a\ne0$, and the triangular system gives unique
$z_3,\ldots,z_n$ with the high coefficients of $P$. The resulting
$\Phi_n$ agrees with $P$ in the coefficients of $T,T^3,\ldots,T^n$, and
both have value $0$ and derivative $a^m$ at $t$. The difference is
$\alpha T^2+\beta$ with $\alpha t^2+\beta=0$ and $2\alpha t=0$, so it
vanishes. Conversely, a source point with $xp\ne0$ gives a root with
$P'(t)=a^m\ne0$, which is simple. Multiple roots receive no points.

\emph{The zero root.} At $x=0$ one has $t=0$, $p=a=1$ and $A_0=0$. The
root $t=0$ of $P$ is always simple, since $P'(0)=1$. By
\eqref{eq:boundary-coefficients}, the value of $A_2$ determines $y$,
and then $C_3,\ldots,C_n$ determine $z_3,\ldots,z_n$. So a zero root
contributes exactly one point, with marking $a=1$. The other $m-1$
markings of $t=0$ have no source point; this is the correction term
in~\eqref{eq:fibre-count}.

\emph{The root at infinity.} At $p=0$, one has $y=-1/x$ and
$1-yT=(T+x)/x$. Direct substitution in~\eqref{eq:compiler} gives
\[
 \Phi_n(T)=-x^{-m}(T+x)^{n-1}+nT-\frac{nx}{2}
                +\frac{x^3z_3}{2}+x^2\sum_{j=4}^n z_jT^{j-2}.
\]
Hence $C_n=0$ and $C_{n-1}=-x^{-m}\ne0$: the image has a simple root
at infinity. Conversely, if $P$ has a simple root at infinity, that is
$q=0\ne r$, then $-x^{-m}=r$ has exactly $m$ solutions $x$. For each,
the coefficient of $T^j$ ($2\le j\le n-2$) determines $z_{j+2}$, the
constant term determines $z_3$, and the coefficient of $T$ is
automatically one. A multiple root at infinity ($q=r=0$) has no lift.

These cases exhaust the source and prove~\eqref{eq:fibre-count}. A
generic $P$ has $n$ simple roots and $P(0)\ne0$, which gives the
geometric degree $nm$.
\end{proof}

The infinity calculation matters: a polynomial of affine degree $n-1$
still has a simple projective root at infinity, and its sheets have
not been lost. All references to a discriminant below mean the
degree-$n$ binary discriminant, including its specializations when the
leading coefficient vanishes.

\subsection{The smallest example}

For $n=3$ we have $m=1$, so there is no marking: a source point is a
root, and $K_3$ has geometric degree three. Expanding
\eqref{eq:compiler} gives
\begin{align*}
 C_3&=(1+xy)\bigl(x^2y^2z_3+2xyz_3+z_3+3xy^3+4y^2\bigr),\\
 A_2&=-\tfrac12\bigl(3x^3y^2z_3+6x^2yz_3+3xz_3
       +9x^2y^3+12xy^2+y\bigr),\\
 A_0&=x-\tfrac32x^2y-\tfrac12x^3z_3,
\end{align*}
and $\det JK_3=\frac12$. Note that $C_3=p(p^2z_3+3xy^3+4y^2)$, in
agreement with the diagonal entry $p^3$. For $n\ge4$ each root carries
$m\ge2$ markings, and the markings are the source of the finer
structure studied below: the index-two phenomenon in the monodromy and
the $m-1$ unramified missing divisors of the normalization.

\section{Polynomial coordinates and descent to dimension three}\label{sec:descent}

Throughout this section $n\ge4$ and $c_2=-m/2$. We want new polynomial
coordinates on the source of $K_n$ in which the lower coefficients
$A_2,C_3,\ldots,C_{n-2}$ are themselves coordinates. Away from $x=0$
this is easy: by~\eqref{eq:coefficient-shape}, $C_k$ contains
$z_{k+2}$ with coefficient $x^2$, so one may solve for $z_{k+2}$ after
inverting $x$. The whole difficulty is to avoid the denominators
$x^{-2}$, that is, to control the divisor $x=0$.

Two facts make this possible. By~\eqref{eq:boundary-coefficients},
along $x=0$ every lower coefficient is a source coordinate plus a
function of $y$. And the coefficient of $z_{k+2}$ in $C_k$ is exactly
$x^2$, with no other occurrence of $z_{k+2}$. Lemma~\ref{lem:completion}
below shows that these two facts are enough to turn $C_k$ into a
coordinate. The coefficient $A_2$ has a slightly different shape and
is treated first, by a unimodular matrix.

\subsection{A polynomial completion lemma}

The model case is $f=u+xu^2+x^2z$ in $\C[x,u,z]$. The map
$u\mapsto u+xu^2$ has inverse $v\mapsto v-xv^2$ modulo $x^2$, and one
checks that $w=(u-(f-xf^2))/x^2$ is a polynomial. Then $(x,f,w)$ are
coordinates. The lemma is the general version of this computation.

\begin{lemma}\label{lem:completion}
Let $R$ be a polynomial ring over $\C$ in auxiliary parameters, let
$c\in\C^*$, and suppose
\[
 f=f_0(x,u)+x^2z\in R[x,u,z],\qquad
 f_0(x,u)\equiv cu+\beta+xh(u)\pmod{x^2},
\]
where $\beta\in R$ and $h\in R[u]$. Put
\[
 \psi_x(v)=\frac{v-\beta}{c}-\frac{x}{c}h\left(\frac{v-\beta}{c}\right),
 \qquad \psi_0(v)=\frac{v-\beta}{c}.
\]
Then
\begin{equation}\label{eq:completion}
 (u,z)\longmapsto\left(f,\frac{u-\psi_x(f)}{x^2}\right)=(v,w)
\end{equation}
is a polynomial automorphism over $R[x]$, with inverse
\[
 u=\psi_x(v)+x^2w,\qquad
 z=\frac{v-f_0(x,\psi_x(v)+x^2w)}{x^2}.
\]
Both Jacobian determinants in the two displayed variables are $-1$.
\end{lemma}

\begin{proof}
Write $g(u)=f_0(x,u)$. We first show that $\psi_x$ inverts $g$ modulo
$x^2$ on both sides. Since $\psi_x(v)\equiv\psi_0(v)\pmod x$,
\[
 g(\psi_x(v))\equiv c\psi_x(v)+\beta+xh\bigl(\psi_0(v)\bigr)
 =v\pmod{x^2}.
\]
Similarly $(g(u)-\beta)/c\equiv u+xh(u)/c\pmod{x^2}$, hence
\[
 \psi_x(g(u))\equiv u+\frac xch(u)-\frac xch(u)=u\pmod{x^2}.
\]
Since $f\equiv g(u)\pmod{x^2}$ and $\psi_x$ has coefficients in $R[x]$,
we get $u-\psi_x(f)\equiv u-\psi_x(g(u))\equiv0$, so $w$ is a
polynomial. Likewise
$v-g(\psi_x(v)+x^2w)\equiv v-g(\psi_x(v))\equiv0\pmod{x^2}$, so the
proposed inverse is polynomial.

The two maps are mutually inverse. Starting from $(u,z)$, we have
$\psi_x(v)+x^2w=u$ by definition of $w$, and then
$(v-g(u))/x^2=z$ because $v=g(u)+x^2z$. Starting from $(v,w)$, we have
$f=g(u)+x^2z=v$, and then $(u-\psi_x(v))/x^2=w$.

Finally, writing $f_u$ for the partial derivative, the forward Jacobian
is
\[
 \det\begin{pmatrix}
 f_u&x^2\\ (1-\psi_x'(f)f_u)/x^2&-\psi_x'(f)
 \end{pmatrix}
 =-f_u\psi_x'(f)-1+\psi_x'(f)f_u=-1.
\]
The inverse determinant is consequently also $-1$.
\end{proof}

The lemma is applied with $u$ a current coordinate, $z$ a source
coordinate not yet used, and all other coordinates in $R$. Its content
is that $f$ only needs to be a coordinate \emph{to first order along
$x=0$}, with constant slope $c$ in $u$; the next-order term $h$ may be
arbitrary.

\subsection{The first coefficient}

By~\eqref{eq:low-coefficients},
\[
 A_2=x\Bigl(-\frac32p^2z_3+xz_4\Bigr)+g_n(x,y),\qquad
 g_n(x,y)=c_2y-\frac{n(n+1)}2xy^2-\frac{3n}{2}x^2y^3.
\]
Lemma~\ref{lem:completion} does not apply directly, because $z_3$ enters
with only one factor of $x$. Instead, note that the row
$(-\frac32p^2,x)$ is unimodular over $\C[x,y]$, since $p^2\equiv1\pmod
x$. It is the first row of
\[
 M(x,y)=\begin{pmatrix}-3p^2/2&x\\2y+xy^2&-2/3\end{pmatrix},
 \qquad \det M=p^2-2xy-x^2y^2=1.
\]
The second row is the simplest completion: $\det M=1$ forces the
$(2,2)$ entry to be $-2/3$ modulo $x$, and then the $(2,1)$ entry must be
$(p^2-1)/x=2y+xy^2$. The coordinate change is built in three steps.
\begin{enumerate}
\item Put $(U,s_4)^\mathsf{t}=M(z_3,z_4)^\mathsf{t}$. Then
$A_2=xU+g_n(x,y)$.
\item Since $g_n-c_2y$ is divisible by $x$, put
$s_3=U+(g_n-c_2y)/x$. Then $A_2=c_2y+xs_3$.
\item Put $v_2=c_2y+xs_3$. Since $c_2\ne0$, this replaces $y$ by
$v_2=A_2$.
\end{enumerate}
Each step is a polynomial automorphism fixing $x,z_5,\ldots,z_n$, with
Jacobian determinants $1$, $1$ and $c_2$. Explicitly, the inverse is
\begin{align}
 Y&=(v_2-xs_3)/c_2,& U&=s_3-(g_n(x,Y)-c_2Y)/x,\notag\\
 y&=Y,&z_3&=-2U/3-xs_4,\label{eq:first-chart}\\
 &&z_4&=-(2Y+xY^2)U-3(1+xY)^2s_4/2,\notag
\end{align}
valid on the entire source, including $x=0$. Along $x=0$ it reads
\begin{equation}\label{eq:first-chart-boundary}
 y\equiv\frac{v_2}{c_2},\qquad
 z_3\equiv-\frac23s_3+\beta_3(v_2),\qquad
 z_4\equiv-\frac32s_4+\beta_4(v_2,s_3)\pmod x,
\end{equation}
with $\beta_3=-\frac{n(n+1)}3(v_2/c_2)^2$ and
$\beta_4=-\frac{2v_2}{c_2}\bigl(s_3+\frac{n(n+1)}2(v_2/c_2)^2\bigr)$.
So along $x=0$ the original $z_3$ and $z_4$ are, up to offsets, the new
coordinates $s_3$ and $s_4$ with the nonzero \emph{slopes} $-2/3$ and
$-3/2$.

For $n=4$ this completes the descent: the source coordinates of
$E_{4,\lambda}$ are $(x,s_3,s_4)$. For example, at $\lambda=0$ one has
$y=xs_3$ and $p=1+x^2s_3$, and, writing $(u,w)=(s_3,s_4)$,
\begin{align*}
 b&=x-\tfrac53x^3u+\tfrac12x^4w+\tfrac{10}3x^5u^2+2x^7u^3,\\
 q&=-\tfrac12\,p\,B,\qquad r=\tfrac13\bigl(2xB-x^2u^2-2u\bigr),\\
 B&=3w+4xu^2+9x^2uw+48x^3u^3+9x^4u^2w+86x^5u^4\\
  &\quad+3x^6u^3w+56x^7u^5+12x^9u^6.
\end{align*}
The map $E_{4,0}=(r,q,b)$ has Jacobian determinant $1$ in the variables
$(x,u,w)$, and geometric degree eight.

\subsection{The uniform induction}

For $n\ge5$ we now straighten $C_3,\ldots,C_{n-2}$ one at a time. The
coordinate $s_k$ below is the current coordinate occupying the slot of
$z_k$; the letters $z_j$ always denote the original coordinate
functions. Congruences modulo $x$ are taken in the ideal generated by
$x$, which is the same ideal in every coordinate system constructed
here, since each change fixes $x$.

\begin{proposition}\label{prop:induction}
Let $n\ge5$ and $3\le k\le n-2$. Suppose that polynomial coordinates
\[
 (x,\ v_2,\ldots,v_{k-1},\ s_k,\ s_{k+1},\ z_{k+2},\ldots,z_n)
\]
on the source have been constructed, with $v_2=A_2$ and $v_j=C_j$,
such that:
\begin{enumerate}
\item[\textup{(a)}] $\C[x,y,z_3,\ldots,z_{k+1}]
 =\C[x,v_2,\ldots,v_{k-1},s_k,s_{k+1}]$;
\item[\textup{(b)}] there are constants $c_k,c_{k+1}\in\C^*$ with
\[
 z_k\equiv c_ks_k+\beta_k,\qquad
 z_{k+1}\equiv c_{k+1}s_{k+1}+\beta_{k+1}\pmod x,
\]
where $\beta_k\in\C[v_2,\ldots,v_{k-1}]$ and
$\beta_{k+1}\in\C[v_2,\ldots,v_{k-1},s_k]$.
\end{enumerate}
Then Lemma~\ref{lem:completion} applies to $f=C_k$, with $u=s_k$ and
$z=z_{k+2}$. The resulting coordinates
\[
 (x,\ v_2,\ldots,v_k,\ s_{k+1},\ s_{k+2},\ z_{k+3},\ldots,z_n),
 \qquad v_k=C_k,\quad s_{k+2}=w,
\]
satisfy \textup{(a)} and \textup{(b)} with $k$ replaced by $k+1$, and
$c_{k+2}=-c_k$.
\end{proposition}

\begin{proof}
\emph{Shape of $C_k$.} By~\eqref{eq:coefficient-shape},
$C_k=f_0+x^2z_{k+2}$, where $f_0$ is a polynomial in $x,y,z_k,z_{k+1}$.
By (a), $f_0\in R_k[x,u]$, where
\[
 R_k=\C[v_2,\ldots,v_{k-1},s_{k+1}],\qquad u=s_k,
\]
and $z_{k+2}$ is an independent coordinate. We apply
Lemma~\ref{lem:completion} over this smaller coefficient ring $R_k$;
the unused coordinates $z_{k+3},\ldots,z_n$ are simply adjoined and
fixed. In particular neither the forward formulas nor the inverse
formulas can introduce dependence on these unused coordinates.

\emph{The slope.} Along $x=0$,~\eqref{eq:boundary-coefficients} gives
$C_k=z_k+\varphi_k(y)$ for an explicit polynomial $\varphi_k$, and
$A_2=c_2y$, so $y\equiv v_2/c_2\pmod x$. By (b),
\[
 f_0\equiv c_ks_k+\beta_k+\varphi_k(v_2/c_2)\pmod x .
\]
This is linear in $u=s_k$ with the constant slope $c_k\ne0$ and offset
$\beta=\beta_k+\varphi_k(v_2/c_2)\in\C[v_2,\ldots,v_{k-1}]\subset R_k$.
Expanding in the independent coordinate $x$ gives
\[
 f_0(x,u)\equiv c_ku+\beta+xh(u)+x^2j(u)\pmod{x^3},
 \qquad h,j\in R_k[u].
\]
Thus the first-order term required by the lemma exists over $R_k$.
The offset $\beta$ is independent of $s_{k+1}$, although $h$ and $j$
may depend on it. The lemma applies with $c=c_k$.

\emph{Condition (a).} The lemma replaces $(s_k,z_{k+2})$ by
$(v_k,s_{k+2})$ over $R_k[x]$, and its formulas involve only $x$,
$v_2,\ldots,v_{k-1}$ and $s_{k+1}$ besides the two exchanged pairs.
Adjoining $z_{k+2}$ to both sides of (a) therefore gives
$\C[x,y,z_3,\ldots,z_{k+2}]=\C[x,v_2,\ldots,v_k,s_{k+1},s_{k+2}]$.

\emph{Condition (b).} The inverse formula gives
$s_k=\psi_x(v_k)+x^2s_{k+2}\equiv\psi_0(v_k)=(v_k-\beta)/c_k\pmod x$.
Substituting into
$\beta_{k+1}$ shows
$z_{k+1}\equiv c_{k+1}s_{k+1}+\beta'_{k+1}$ with
$\beta'_{k+1}\in\C[v_2,\ldots,v_k]$. For the new slot, Taylor expansion
in the second argument gives
\[
 f_0\bigl(x,\psi_x(v_k)+x^2s_{k+2}\bigr)
 \equiv f_0\bigl(x,\psi_x(v_k)\bigr)+x^2s_{k+2}\,\partial_uf_0(0,\psi_0(v_k))
 \pmod{x^3},
\]
and $\partial_uf_0(0,\cdot)=c_k$. Hence the inverse formula for $z$
yields
\[
 z_{k+2}\equiv -c_k\,s_{k+2}+\beta_{k+2}\pmod x,\qquad
 \beta_{k+2}=\Bigl[\frac{v_k-f_0(x,\psi_x(v_k))}{x^2}\Bigr]_{x=0}.
\]
More explicitly, writing $\alpha=(v_k-\beta)/c_k$ and expanding
$f_0(x,\alpha-xh(\alpha)/c_k)$ to order two gives
\begin{equation}\label{eq:induction-offset}
 \beta_{k+2}=\frac{h'(\alpha)h(\alpha)}{c_k}-j(\alpha)
 \in\C[v_2,\ldots,v_k,s_{k+1}],
\end{equation}
where the prime denotes differentiation in $u$. Thus the new offset
has precisely the allowed dependence, is independent of $s_{k+2}$,
and the new slope is the nonzero constant $c_{k+2}=-c_k$.
\end{proof}

By~\eqref{eq:first-chart} and~\eqref{eq:first-chart-boundary}, the
hypotheses hold for $k=3$, with $c_3=-2/3$ and $c_4=-3/2$. Applying Proposition~\ref{prop:induction} for
$k=3,\ldots,n-2$ produces coordinates
\[
 (x,\ A_2,C_3,\ldots,C_{n-2},\ s_{n-1},\ s_n),
\]
and the slopes used along the way are
\[
 c_3,c_4,c_5,c_6,\ldots=-\tfrac23,\ -\tfrac32,\ \tfrac23,\ \tfrac32,\
 -\tfrac23,\ldots,
\]
all nonzero constants. Every change fixes $x$ and is defined over the
polynomial ring in the coefficient coordinates obtained so far. Keeping
those coefficients as variables, rather than substituting numerical
values, gives a polynomial automorphism of the entire source. The
offsets are explicit but grow quickly; only their dependence matters.
Because these are identities over the polynomial ring in the coefficient
coordinates, specializing any prefix to any complex values preserves
both inverse identities. No exceptional parameter values are excluded.

\begin{proof}[Proof of Theorem~\ref{thm:descent-main}]
Let $\Theta$ be the inverse of the composite coordinate change
constructed above, so that it expresses the original source coordinates
in the new ones. In these coordinates the lower coefficients are coordinate
functions, so, after reordering source and target coordinates,
\[
 K_n\circ\Theta:(\lambda,\xi)\longmapsto(\lambda,E_{n,\lambda}(\xi)),
 \qquad \xi=(x,s_{n-1},s_n).
\]
Its Jacobian matrix is block lower triangular, with an identity block
for $\lambda$ and the block $\partial E_{n,\lambda}/\partial\xi$. Since
$\det JK_n=m/2$ and $\det J\Theta$ is a nonzero constant, so is
$\det\partial E_{n,\lambda}/\partial\xi$, for every $\lambda$.

The sign can be tracked. Each change keeps its new coordinates in the
slots of the old ones, so in the source order
$(x,v_2,\ldots,v_{n-2},s_{n-1},s_n)$ the automorphism $\Theta^{-1}$ has
determinant $c_2(-1)^{n-4}$, and $K_n\circ\Theta$, with outputs in the
order of $K_n$, has determinant $(m/2)/(c_2(-1)^n)=(-1)^{n+1}$. Moving
$A_2$ to the front of the outputs contributes $(-1)^{n-2}$, and moving
$x$ behind the coefficient coordinates contributes $(-1)^{n-3}$. Hence,
with $\xi=(x,s_{n-1},s_n)$ and output order $(C_{n-1},C_n,A_0)$,
\[
 \det\frac{\partial E_{n,\lambda}}{\partial\xi}=(-1)^n .
\]

Because $\Theta$ is bijective, the fibre of $E_{n,\lambda}$ over
$(r,q,b)$ is in bijection with the fibre of $K_n$ over the
polynomial~\eqref{eq:pencil}, and~\eqref{eq:fibre-count} applies
unchanged. For every fixed $\lambda$ there are polynomials
in~\eqref{eq:pencil} with $q\ne0$, $b\ne0$ and $n$ distinct roots:
fix $q\ne0$ and $r$, and choose $b\ne0$ outside the finitely many
critical values of $qT^n+rT^{n-1}+\sum\lambda_jT^j+T$. Such targets form
a nonempty Zariski open set, over which the fibre has $nm$ points.
Hence the geometric degree is $nm$ for every parameter. Stopping the
induction after fixing any prefix of the lower coefficients gives the
same conclusion for the intermediate descents.
\end{proof}

The recursive formulas specify the maps without expanding their
coordinates into large monomial lists. They also explain why a generic
rational elimination would not suffice: solving $C_k=\lambda_k$ for
$z_{k+2}$ introduces the denominator $x^2$, and the completion lemma is
precisely what removes it.

\section{Exact generic monodromy}\label{sec:monodromy}

The generic monodromy of a dominant, generically finite map is the
Galois group of the normal closure of the source function field over
the target function field, acting on the embeddings of the source
field. Equivalently, it is the monodromy of the finite covering over a
sufficiently small nonempty Zariski open set of the target.

\subsection{Overview}

Let $k$ be the target function field of $K_n$ or of one of its
descents, and let $P\in k[T]$ be the generic target polynomial. By
Section~\ref{sec:root}, the source function field is $k(t,a)$, where
$t$ is a root of $P$ and $a^m=P'(t)$: indeed $x=t/a$ and $y=(1-a)/t$
by~\eqref{eq:marked-chart}, and the $z_j$ are then rational in $x,y$
and the target coefficients by the triangular system. The $nm$ sheets
over a generic point are the pairs $(t_i,\zeta a_i)$, where
$t_1,\ldots,t_n$ are the roots, $a_i$ is one chosen $m$-th root of
$d_i=P'(t_i)$, and $\zeta\in\mu_m$.

A monodromy element permutes the roots and then rotates each marking
by a root of unity. This gives the upper bound $\mu_m^n\rtimes S_n$,
the wreath product $\mu_m\wr S_n$. The question is which rotations
actually occur. By Kummer theory, this is governed by multiplicative
relations among the $d_i$ modulo $m$-th powers. We detect such relations
with valuations at which exactly two roots collide
(Lemma~\ref{lem:pair}), and we produce these valuations, together with
the full symmetric group on the roots, from a single line in the target
space (Lemma~\ref{lem:morse}). The only relation that survives comes
from the classical identity~\eqref{eq:derivative-product} between the
product of the $d_i$ and the discriminant. It exists exactly when $m$
is even, and it cuts the group down by a factor of two.

\subsection{Pair valuations and the derivative-root extension}

\begin{lemma}\label{lem:pair}
Let $k$ be a function field over $\C$ and
$P=q\prod_{i=1}^n(T-t_i)\in k[T]$, with splitting field $\tilde k$ and
$\Gal(\tilde k/k)=S_n$. Put $m=n-2$. Suppose $\tilde k$ has a discrete
valuation, trivial on $\C$, at which $q$ is a unit, one root difference
has value one, and every other root difference has value zero. Let
$L=\tilde k(a_1,\ldots,a_n)$ with $a_i^m=P'(t_i)$. Then $L/k$ is Galois, and
$\Gal(L/k)$, acting on the $nm$ pairs $(t_i,\zeta a_i)$, is the group
$G_n$ of~\eqref{eq:group}. This action is transitive, so the field
$k(t_1,a_1)$ has degree $nm$ over $k$ and $L$ is its normal closure.
\end{lemma}

\begin{proof}
\emph{Possible relations.} Set $d_i=P'(t_i)=q\prod_{j\ne i}(t_i-t_j)$.
At the given valuation, if $t_1-t_2$ is the colliding difference, then
$d_1$ and $d_2$ have value one and every other $d_i$ has value zero.
Since $\Gal(\tilde k/k)=S_n$ is transitive on unordered pairs, composing the
valuation with Galois automorphisms gives, for each pair $\{i,j\}$, a
valuation at which $d_i,d_j$ have value one and the others value zero.
Now consider the group of relations
\[
 \mathcal R=\left\{(\ell_i)\in(\Z/m)^n:\prod_i d_i^{\ell_i}\in\tilde k^{\times m}\right\}.
\]
An $m$-th power has value divisible by $m$, so every relation satisfies
$\ell_i+\ell_j\equiv0\pmod m$ for every pair. Comparing three distinct indices
shows that all $\ell_i$ are equal and $2\ell_i\equiv0$. When $m$ is odd only
the zero relation is possible. When $m$ is even there is at most one
further relation, namely $(m/2,\ldots,m/2)$.

\emph{The relation exists for even $m$.} Let
$\delta=\prod_{i<j}(t_i-t_j)$. Expanding the definition of $d_i$ gives
\begin{equation}\label{eq:derivative-product}
 \prod_i d_i=q^n\prod_i\prod_{j\ne i}(t_i-t_j)=(-1)^{n(n-1)/2}q^n\delta^2 .
\end{equation}
When $m$ is even, $n$ is even too, and the $(m/2)$-th power of the right
side equals $\pm(q^{n/2}\delta)^m$, an $m$-th power in $\tilde k$ because complex
constants have all roots. So $|\mathcal R|=\gcd(2,m)$.

\emph{The degree.} The field $\tilde k$ contains $\mu_m$, so Kummer theory
\cite[Theorem 10.3.7]{SharifiKummerDuality} identifies $\Gal(L/\tilde k)$
with the dual of the subgroup of $\tilde k^\times/\tilde k^{\times m}$
generated by the
$d_i$. That subgroup is $(\Z/m)^n/\mathcal R$, so
\[
 [L:\tilde k]=\frac{m^n}{|\mathcal R|}=\frac{m^n}{\gcd(2,m)}.
\]

\emph{The group.} The set $\{d_1,\ldots,d_n\}$ is $\Gal(\tilde k/k)$-stable,
so $L$ is Galois over $k$, and every automorphism sends
$a_i\mapsto\zeta_ia_{\sigma(i)}$ for some $\zeta_i\in\mu_m$ and
$\sigma\in S_n$. Thus $\Gal(L/k)$ embeds in $\mu_m^n\rtimes S_n$, with
order $n!\,m^n/\gcd(2,m)$. For odd $m$ this is the whole group. For even
$m$, consider
\[
 B=\frac{\prod_i a_i^{m/2}}{q^{n/2}\delta}.
\]
By~\eqref{eq:derivative-product}, $B^2=(-1)^{n(n-1)/2}$, so $B\in\C$ is
fixed by every automorphism. The element $(\zeta;\sigma)$ multiplies
the numerator by $(\prod_i\zeta_i)^{m/2}$ and $\delta$ by $\sgn(\sigma)$, so
it fixes $B$ exactly when $(\prod_i\zeta_i)^{m/2}=\sgn(\sigma)$. These
elements form a subgroup of index two, the kernel of a surjective
character, whose order equals the degree already computed. Hence
$\Gal(L/k)=G_n$.

\emph{Transitivity.} The permutations $\sigma$ move any root block to
any other. Within one block, the element with $\zeta_1=\zeta$,
$\zeta_2=\zeta^{-1}$, all other $\zeta_i=1$ and $\sigma=1$ lies in $G_n$
and rotates the first marking by $\zeta$. So $G_n$ is transitive on
the $nm$ pairs. The conjugates of $(t_1,a_1)$ generate $L$, which is
therefore the normal closure of $k(t_1,a_1)$. For $m=1$ the lemma
simply describes the root splitting field.
\end{proof}

\subsection{A Morse line for every fixed coefficient tuple}

A polynomial $g\in\C[T]$ is called \emph{Morse} if its critical points
are simple and its critical values are pairwise distinct. The
branched cover $g:\A^1_T\to\A^1_b$ of such a polynomial has local
monodromy a transposition around each critical value, and its monodromy
group is the full symmetric group
\cite[Theorem 4.4.5]{Serre1992TopicsGalois}; the argument is recalled
below.

\begin{lemma}\label{lem:morse}
For any $\lambda_2,\ldots,\lambda_{n-2}\in\C$, the polynomial
\eqref{eq:pencil} over $\C(r,q,b)$ has root group $S_n$ and a
valuation satisfying the hypotheses of Lemma~\ref{lem:pair}.
\end{lemma}

\begin{proof}
\emph{A Morse member.} Write $P_0(T)=T+\sum_{j=2}^{n-2}\lambda_jT^j$ for
the fixed part of~\eqref{eq:pencil}. Set $r=0$, $q=\varepsilon^{-(n-1)}$,
and $T=\varepsilon z$. Then
\[
 \frac{qT^n+P_0(T)}{\varepsilon}
 =z^n+z+\sum_{j=2}^{n-2}\lambda_j\varepsilon^{j-1}z^j.
\]
At $\varepsilon=0$ this is $z^n+z$. Its $n-1$ critical points satisfy
$z^{n-1}=-1/n$; they are simple, and their critical values
$z(z^{n-1}+1)=(n-1)z/n$ are distinct. Both properties are open, so
they persist for sufficiently small nonzero $\varepsilon$. Fix such a
value of $q$. Then $g(T)=qT^n+P_0(T)$ is Morse.

\emph{The root group.} Along the line $\{r=0,\ q\text{ fixed}\}$ in the
target, the roots of~\eqref{eq:pencil} are the points of
$g^{-1}(b)$. The fundamental group of the $b$-line minus the critical
values is generated by loops around them, whose monodromy consists of
transpositions. Since $\A^1_T$ is connected, the generated group is
transitive. A transitive subgroup of $S_n$ generated by transpositions
is $S_n$. The monodromy along the line is a subgroup of the root
monodromy over the whole target, which is therefore $S_n$ as well.

\emph{The valuation.} At a critical value $b_0$, the polynomial
$g(T)-b_0$ has exactly one double root and $n-2$ simple roots. Along the
$b$-line the discriminant of $g(T)-b$ is a constant multiple of
$\prod(b-b_j)$ over the critical values $b_j$, so it has a simple zero
at $b_0$. Hence the discriminant hypersurface $\Delta$ is smooth at this
point and transverse to the line, and its divisorial valuation $v_\Delta$
on $k$ satisfies $v_\Delta(\disc P)=1$. Modulo $v_\Delta$, $P$ has one
double root $t_0$ and $n-2$ simple roots. Over the completion $\hat k$ of
$k$ at $v_\Delta$, Hensel's lemma therefore factors $P=Q\tilde R$, with
$Q$ quadratic reducing to a multiple of $(T-t_0)^2$ and $\tilde R$
separable modulo $v_\Delta$ and coprime to $T-t_0$. The roots of
$\tilde R$ generate an unramified extension, and their residues are
distinct from each other and from $t_0$. Since the resultant of $Q$ and
$\tilde R$ and the discriminant of $\tilde R$ are units,
$\disc Q$ has value one. It is not a square in $\hat k$, so $Q$ is
irreducible and generates a ramified quadratic extension. Hence any
extension $w$ of $v_\Delta$ to $\tilde k$, normalized with value group
$\Z$, has ramification index two, and $w\bigl((t_1-t_2)^2\bigr)=2$ for
the two roots $t_1,t_2$ of $Q$. So $t_1-t_2$ is a uniformizer, while $q$
and all other root differences are units. This is the required
valuation.
\end{proof}

\begin{proof}[Proof of Theorem~\ref{thm:monodromy-main}]
For a completely fixed tuple, the target field is $\C(r,q,b)$, and
Lemmas~\ref{lem:morse} and~\ref{lem:pair} give the group $G_n$. For a
partial descent or the original map, the target has additional free
lower coefficients. Fix them arbitrarily to obtain the same Morse
line. Its root monodromy still gives $S_n$, and its transverse simple
branch point gives a pair valuation for the larger target field. The
upper bound from~\eqref{eq:derivative-product} holds over every base
field. Thus the group is $G_n$ in every case.
\end{proof}

\begin{center}
\begin{tabular}{rrrr}
\toprule
$n$ & $m$ & Number of sheets & $|G_n|$\\
\midrule
3&1&3&6\\
4&2&8&192\\
5&3&15&29160\\
6&4&24&1474560\\
\bottomrule
\end{tabular}
\end{center}

For $n\ge4$, the $n$ root blocks of size $m$ form a nontrivial system
of imprimitivity. This is used in Section~\ref{sec:comparison}. An
arbitrary nonlinear specialization can change both the root group and
the derivative-root relations; the Morse-line proof establishes
preservation only in the stated coefficient families.

\section{The finite normalization and its affine open}\label{sec:normalization}

\subsection{Overview}

The source of a non-invertible Keller map is not finite over its
target: sheets escape to infinity over the nonproperness locus. The finite
normalization $Z$ adds these escaped sheets back as divisors. Here $Z$
has a transparent description: its points are \emph{all} pairs
consisting of a projective root of the target polynomial and a
marking, with no simplicity condition. Two charts suffice, one for
finite roots and one near the root at infinity, and the second is
already an affine space. The source is the open set of pairs with a
simple root, except that the root $t=0$ keeps only the marking $a=1$.
So the missing divisors are the critical divisor, where the root is
multiple, and the $m-1$ divisors $\{t=0,\ a=\zeta\}$ with $\zeta\ne1$.

Fix a successive coefficient slice $Y\simeq\A^d$. In~\eqref{eq:pencil}
the symbols $\lambda_j$ now denote either fixed complex numbers or the
remaining independent coordinates. Let $I\subset Y\times\PP^1$ be the
incidence variety of projective roots of the degree-$n$ homogenization
of $P$, with homogeneous coordinates $[X_1:X_0]$ and $t=X_1/X_0$. Since
the coefficient of $X_1X_0^{n-1}$ is one, that binary form never vanishes
identically. Thus $I\to Y$ is finite flat of degree $n$:
it is a relative degree-$n$ effective divisor on $\PP^1_Y$.

\subsection{Two charts and smoothness}

On the finite-root chart define the cover by
\[
 P(t)=0,\qquad a^m=P'(t).
\]
On the infinity chart put $Q(u)=u^nP(1/u)$ and define it by
\[
 Q(u)=0,\qquad v^m=-Q'(u).
\]
On the overlap $u=1/t$, differentiating $Q(u)=u^nP(1/u)$ and using
$P(1/u)=0$ gives
\begin{equation}\label{eq:gluing}
 -Q'(u)=u^{n-2}P'(1/u),\qquad\text{so we glue by}\quad v=ua.
\end{equation}
These formulas define a finite flat degree-$m$ cover $Z\to I$.
Consequently $\pi:Z\to Y$ is finite flat of degree $nm$.
The picture, with generic degrees indicated, is
\[
 \A^d\simeq U\lhook\joinrel\longrightarrow Z
       \xrightarrow{\ m\ } I\xrightarrow{\ n\ }Y.
\]
The open immersion on the left is identified below.

Write
\[
 Q_0(u)=\sum_{j=2}^{n-2}\lambda_j u^{n-j}+u^{n-1}-bu^n,
\]
so that $Q=q+ru+Q_0$. The two equations of the infinity chart can be
solved for $r$ and $q$:
\begin{equation}\label{eq:infinity-chart}
 r=-v^m-Q_0'(u),\qquad q=uv^m+uQ_0'(u)-Q_0(u).
\end{equation}
Hence the infinity chart is an affine space $V\simeq\A^d$, with
coordinates $u,v,b$ and the free lower coefficients.

The complement of $V$ in $Z$ lies over the root $t=0$, where
$-b=P(0)=0$ and $a^m=P'(0)=1$. At these points the derivatives of the two
equations $P(t)=0$ and $a^m-P'(t)=0$ with respect to $(b,a)$ form a
triangular matrix with diagonal $(-1,ma^{m-1})$, which is invertible.
So $Z$ is smooth there, with local coordinates given by $t$ and the
remaining target coordinates. In particular $t=0$ is a divisor near
every point of $Z\setminus V$, so this complement contains no component
of $Z$. Thus $V$ is dense. Since $V$ is integral and $Z$ is smooth,
$Z$ is integral too.

By~\eqref{eq:marked-chart}, the function field of $Z$ is the source
field of the polynomial map. A finite normal model with this field is
the normalization of $Y$ in that extension. This proves the first
assertion of Theorem~\ref{thm:normal-main}.

Holding $b$ and the free lower coefficients fixed, the chart
determinant is especially simple:
\begin{equation}\label{eq:chart-det}
 \det\frac{\partial(r,q)}{\partial(u,v)}=m v^{2m-1}.
\end{equation}
It follows by differentiating~\eqref{eq:infinity-chart}; the two terms
involving $uQ_0''$ cancel.

\subsection{The source open and every missing divisor}

\emph{The source embeds in $Z$.} The original source maps to $I$ by
$[x:p]$, whose entries never vanish together. Its lifts to $Z$ are
$a=1/p$ where $p\ne0$, and $v=1/x$ where $x\ne0$; they agree
by~\eqref{eq:gluing}. The chart equations hold by~\eqref{eq:root-derivative}
where $xp\ne0$, and hence everywhere by continuity. This morphism is
birational, and quasi-finite since its fibres are contained in fibres
of the Keller map. Zariski's Main Theorem, with the normal target $Z$,
makes it an open immersion \cite{StacksZariskiMain}. The same argument
applies after the polynomial coordinate descents. Denote its image by
$U$.

\emph{The part inside $V$.} Put $V^\circ=V\setminus\{v=0\}$. We show
$U\cap V=V^\circ$. The chart $V$ is the locus where the root is not
$[0:1]$, so a source point in $V$ has $x\ne0$, hence $v=1/x\ne0$.
Conversely, let $(u,v,b,\ldots)\in V^\circ$. The reconstruction is
\begin{equation}\label{eq:uv-reconstruction}
 x=1/v,\qquad y=u-v,
\end{equation}
since $u=p/x=1/x+y$. At fixed $(x,y)$, the matrix from
$(z_3,\ldots,z_n)$ to $(A_0,A_2,C_3,\ldots,C_{n-2})$ is lower triangular
by~\eqref{eq:coefficient-shape} and~\eqref{eq:low-coefficients}, with
diagonal $(-x^3/2,x^2,\ldots,x^2)$ and determinant
$-\frac12x^{2n-3}$. (For $n=3$ the list is just $A_0$.) Since $x\ne0$,
these coefficients determine $z_3,\ldots,z_n$ without dividing by $p$.
The resulting source point has the prescribed $b$ and lower
coefficients, and its root $[x:p]$ has chart coordinates $(u,v)$; by
\eqref{eq:infinity-chart} its remaining coefficients $r,q$ agree too.
This includes the points $u=0$, which are the simple roots at infinity.

\emph{The part outside $V$.} The remaining points have $t=0$,
$a=\zeta\in\mu_m$ and $b=0$. Source points with $t=0$ have $x=0$ and
$a=1/p=1$, so only the branch $a=1$ can be retained. It is: near that
branch,
\[
 (a-1)(1+a+\cdots+a^{m-1})=P'(t)-1,
\]
and $P'(t)-1$ is divisible by $t$, while the factor in parentheses is a
unit near $a=1$. So $y=(1-a)/t$ is regular there, as are $x=t/a$ and,
by the unit determinant $p^{2n-3}$, the $z_j$. On every branch
$a=\zeta\ne1$, the function $y=(1-a)/t$ has a simple pole, so these
branches are omitted.

It follows that the whole omitted boundary is
\begin{equation}\label{eq:boundary}
 Z\setminus U=E_{\mathrm{crit}}\ \sqcup\!
 \bigsqcup_{\zeta\in\mu_m\setminus\{1\}} E_\zeta,
 \qquad E_{\mathrm{crit}}=\{v=0\}\subset V,
 \quad E_\zeta=\{t=0,a=\zeta\}.
\end{equation}
On $V$, the equations $Q(u)=0=Q'(u)$ show that $E_{\mathrm{crit}}$ is
the locus of multiple projective roots. Each component is $\A^{d-1}$:
$E_{\mathrm{crit}}$ is a coordinate hyperplane of $V$, and $E_\zeta$
maps isomorphically onto the hyperplane $\{b=0\}$ of $Y$. The
components are pairwise disjoint, because $P'(0)=1$ and so the zero
root is never multiple. The retained zero branch $a=1$ is another copy
of $\{b=0\}$ inside $U$.

\subsection{Ramification and fibre defects}

Let $\Delta\subset Y$ be the reduced binary discriminant. The image of
$E_{\mathrm{crit}}$ consists of precisely the polynomials with a
multiple projective root. Thus $\Delta$ is irreducible, as the finite
image of the irreducible $E_{\mathrm{crit}}$. The Morse line in
Lemma~\ref{lem:morse} shows that its generic member has just one double
root. Hence $E_{\mathrm{crit}}\to\Delta$ is finite and birational, and,
since $E_{\mathrm{crit}}$ is smooth, it is the normalization of
$\Delta$.

At a generic double root the incidence map $I\to Y$ has ramification
index two. The cyclic cover has index $m$ there, because its defining
derivative has a simple zero on $I$. Therefore $E_{\mathrm{crit}}$ has
ramification index $2m$ and residue degree one over $\Delta$,
consistent with~\eqref{eq:chart-det}, whose zero of order $2m-1$ along
$v=0$ is the ramification divisor. At each $E_\zeta$, the derivatives
with respect to $(t,a)$ in the finite chart form the invertible matrix
$\bigl(\begin{smallmatrix}1&0\\-P''(0)&m\zeta^{m-1}\end{smallmatrix}\bigr)$,
so the map is unramified there and has residue degree one.

Over a general point of $\Delta$, $2m$ sheets are missing from $U$,
counted with multiplicity. Over a general point of $b=0$, $m-1$ sheets
are missing, although those omitted branches are unramified. Thus
ramification alone cannot detect the affine-source defect. A simple
root at infinity is retained, so $q=0$ is not an additional boundary
image. For $m=1$ there are no omitted zero-root branches.

\subsection{Picard group and homotopy type}

Since $U\simeq\A^d$ is factorial, any divisor on $Z$ becomes principal
on $U$. Subtracting that principal divisor leaves a divisor supported
on~\eqref{eq:boundary}. If a linear combination of these boundary
divisors were principal, its rational function would restrict to a
unit on $U$, hence a constant, and the combination would be zero.
Smoothness identifies Weil and Cartier class groups, proving
\begin{equation}\label{eq:pic}
 \Pic(Z)=\bigoplus_{E\subset Z\setminus U}\Z[E]\simeq\Z^m.
\end{equation}

The cyclic action $a\mapsto\zeta a$, $v\mapsto\zeta v$ preserves the
chart equations and the gluing, so it acts on $Z$ over $I$ and preserves
$V$ and $V^\circ$. Write
\[
 D_\zeta=\{t=0,a=\zeta\}\qquad(\zeta\in\mu_m),
\]
including the retained branch $D_1$. The source reconstruction above
gives the set-theoretic equality $U=V^\circ\cup D_1$. Hence each
translate $U_\zeta=\zeta\cdot U$ is an open copy of $\A^d$ with
\[
 U_\zeta=V^\circ\cup D_\zeta,\qquad
 V\cap U_\zeta=V^\circ,\qquad
 U_\zeta\cap U_\eta=V^\circ\quad(\zeta\ne\eta).
\]
Since $Z\setminus V=\bigsqcup_\zeta D_\zeta$, the $m+1$ affine open
sets $V$ and $U_\zeta$ cover $Z$. Every intersection of at least two
distinct charts is the same open set
\[
 V^\circ\simeq\Gm\times\A^{d-1}.
\]

We compute the homotopy type in the complex topology. Enumerate
$\mu_m=\{\zeta_1,\ldots,\zeta_m\}$ and put
\[
 X_0=V(\C),\qquad
 X_j=V(\C)\cup\bigcup_{i=1}^j U_{\zeta_i}(\C),\qquad
 A=V^\circ(\C)\simeq S^1.
\]
Each $X_j$ is an open submanifold of the smooth complex variety
$Z(\C)$, hence is paracompact and has the homotopy type of a CW
complex. The two sets $X_{j-1}$ and $U_{\zeta_j}(\C)$ form an open
cover of $X_j$, with intersection exactly $A$. By the open-cover
gluing theorem \cite[Proposition 4G.2]{Hatcher2002AlgebraicTopology},
the double mapping cylinder of
\[
 X_{j-1}\longleftarrow A\longrightarrow U_{\zeta_j}(\C)
\]
maps by a homotopy equivalence onto $X_j$. Thus this union is a
homotopy pushout; no cofibration assertion for the open inclusions
is needed.

The new chart is contractible, so this homotopy pushout is homotopy
equivalent to the mapping cone of $A\hookrightarrow X_{j-1}$.
That inclusion factors as $A\hookrightarrow V(\C)\hookrightarrow
X_{j-1}$ and is therefore null-homotopic. Indeed, in affine
coordinates on $V(\C)\simeq\C^d$, straight-line contraction to a fixed
point of $V(\C)$ supplies such a homotopy inside $X_{j-1}$. The
mapping cone of a null-homotopic map is homotopy equivalent to
$X_{j-1}\vee\Sigma A$, where $\Sigma$ denotes suspension. Since
$A\simeq S^1$ is connected, $\Sigma A\simeq S^2$. Consequently
\[
 X_j\simeq X_{j-1}\vee S^2\qquad(1\le j\le m).
\]
Starting with the contractible $X_0$ and using $X_m=Z(\C)$ gives
\begin{equation}\label{eq:homotopy}
 Z(\C)\simeq\bigvee_{j=1}^m S^2.
\end{equation}
This also covers $m=1$: the two charts then give a single sphere.
In particular $\pi_1(Z)=1$, $H_2(Z,\Z)=\Z^m$, and all other reduced
integral homology groups vanish. This is consistent with~\eqref{eq:pic}.

\subsection{Rational and polynomial deck groups}

For a transitive monodromy action on a generic fibre $\Omega$, the
automorphism group of the source field over the target field is the
centralizer of the monodromy group in $\mathrm{Sym}(\Omega)$. A
permutation commuting with the monodromy is determined by the image of
one point, and that image must be fixed by the stabilizer of the point.
In $G_n$, the stabilizer of a pair $(t_1,a_1)$ consists of the elements
with $\sigma(1)=1$ and $\zeta_1=1$. It fixes the $m$ points
$(t_1,\zeta a_1)$ of the same root block, and no other point, since it
contains elements moving $t_j$ to $t_k$ for any $j,k\ne1$ (for even $n$,
adjust some $\zeta_k$ with $k\ne1$ to satisfy the sign condition). Thus
the centralizer has at most $m$ elements. The global cyclic action
already gives $m$, so the rational deck group is exactly $\mu_m$.

Every nonidentity element moves the retained zero-root branch $a=1$
to an omitted branch $a=\zeta$, so none preserves $U$. A polynomial deck
automorphism of the source would induce one of these field
automorphisms, extend to $Z$ by normality, and preserve $U$. So only
the identity is possible. This finishes the proof of
Theorem~\ref{thm:normal-main}.

\begin{remark}
For $n=4$ we have $m=2$, and the picture is as small as possible:
$Z\setminus U=E_{\mathrm{crit}}\sqcup E_{-1}$, where $E_{-1}$ is the
branch of the root $t=0$ with marking $a=-1$. Thus $\Pic(Z)\simeq\Z^2$
and $Z(\C)\simeq S^2\vee S^2$, and the rational deck transformation
$a\mapsto-a$ exchanges the retained zero branch with $E_{-1}$.
\end{remark}

\begin{remark}
Triviality of the polynomial deck group holds for every complex
affine-space Keller map. That group embeds in the automorphism group of
a finite field extension, so is finite. A deck automorphism fixing a
point is the identity near it by local analytic invertibility, hence
everywhere, so the group acts freely. For the quotient, multiplicativity
of the compactly supported Euler characteristic gives
$1=\chi_c(\A^d)=|G|\chi_c(\A^d/G)$, forcing $|G|=1$. Our boundary
calculation describes explicitly why the nontrivial \emph{rational}
deck transformations fail to preserve the source.
\end{remark}

\section{Sections by subvarieties and factorized sectors}\label{sec:sectors}

Restricting a Keller map to the preimage of a subvariety of its target
is the natural way to look for lower-dimensional examples. If
$\Gamma\subset Y$ is a smooth subvariety, the preimage $F^{-1}(\Gamma)$
is smooth and étale over $\Gamma$, but it can have several components,
none of which need be an affine space. This section describes these
components, which we call \emph{sectors}, exactly. Everything follows
from Theorem~\ref{thm:normal-main} by base change.

Throughout, $F$, $Y$ and $U\subset Z$ are as in
Theorem~\ref{thm:normal-main}, with $F=K_n$ and $Y=\A^n$ the full
target, and $m=n-2$. For a point $\eta\in Y$ write $P_\eta$ for its
target polynomial and
$P_{\eta,h}(X_1,X_0)=X_0^nP_\eta(X_1/X_0)$ for its degree-$n$
homogenization, so that $t=X_1/X_0$ and a source point has root
$[X_1:X_0]=[x:p]$. We write $\partial_1,\partial_0$ for the partial
derivatives in $X_1,X_0$. The idea of the section is simple. A sector
over $\Gamma$ selects, over each point of $\Gamma$, one simple root of
the target polynomial together with a marking. If the target polynomial
factors over $\Gamma$, the selected root belongs to one factor $G$, and
the marking is then controlled by $G$ alone.

\subsection{Base change}

\begin{proposition}[Preimages of subvarieties]\label{prop:base-change}
Let $\Gamma$ be a variety with a morphism $g:\Gamma\to Y$. Put
$Z_\Gamma=Z\times_Y\Gamma$ and $U_\Gamma=U\times_Y\Gamma$, the preimage
of $\Gamma$ in the source. Then $U_\Gamma$ is open in $Z_\Gamma$. A
point of $Z_\Gamma$ over $\gamma\in\Gamma$, written
$(\gamma,[t],a)$ in the finite chart or $(\gamma,u,v)$ in the infinity
chart, lies in $U_\Gamma$ if and only if
\begin{enumerate}
\item $[t]$ is a simple projective root of $P_{g(\gamma)}$, and
\item $a=1$ when $t=0$.
\end{enumerate}
If $\Gamma$ is smooth, $U_\Gamma$ is smooth, and its irreducible and
connected components coincide.
\end{proposition}

\begin{proof}
Base change of the open immersion $U\hookrightarrow Z$ along $g$ is an
open immersion. Its complement is the base change of
$Z\setminus U=E_{\mathrm{crit}}\sqcup\bigsqcup_{\zeta\ne1}E_\zeta$
from~\eqref{eq:boundary}.

The charts of $Z$ are cut out by equations in the coefficients of $P$:
\begin{itemize}
\item on the finite chart, $P(t)=0$ and $a^m=P'(t)$;
\item on the infinity chart, $Q(u)=0$ and $v^m=-Q'(u)$, where
$Q(u)=P_h(1,u)$.
\end{itemize}
So $Z_\Gamma$ is given by the same equations with $P=P_{g(\gamma)}$.
The divisor $E_{\mathrm{crit}}$ is $\{v=0\}$. It contains the finite
points with $a=0$, since $v=ua$, and these are exactly the multiple
projective roots. Each $E_\zeta$ is $\{t=0,\ a=\zeta\}$.

Finally, $U_\Gamma\to\Gamma$ is the base change of the étale morphism
$F$, so it is étale. Hence $U_\Gamma$ is smooth when $\Gamma$ is, and
then its connected components are irreducible.
\end{proof}

\subsection{Factorizations}

Let $\Gamma$ be smooth and irreducible, and suppose that
$P_{g(\gamma),h}=G_hH_h$ in $\mathcal O(\Gamma)[X_1,X_0]$, with binary
forms of degrees $e$ and $f=n-e$ and $G_h$ irreducible over
$\C(\Gamma)$. We call $H_h$ the \emph{cofactor}. Write $G(T)=G_h(T,1)$,
$H(T)=H_h(T,1)$, $Q_G(u)=G_h(1,u)$ and $Q_H(u)=H_h(1,u)$, so that
$Q=Q_GQ_H$. Let $I_G=\{G_h=0\}\subset\Gamma\times\PP^1$. Since $P_h$
has coefficient one at $X_1X_0^{n-1}$, neither factor vanishes
identically at any point. So $I_G$ is a relative effective divisor,
finite and flat over $\Gamma$. It is reduced, being Cohen--Macaulay and
generically reduced, and irreducible, because every component dominates
$\Gamma$ and the generic fibre is irreducible.

An irreducible component $S$ of $U_\Gamma$ dominating $\Gamma$
\emph{belongs to $G$} if $G_h(x,p)=0$ on $S$. By
Proposition~\ref{prop:base-change}, the generic root of every dominating
component is simple, so each such component belongs to exactly one
irreducible factor.

The key input is a homogeneous form of the marking equation. Where
$xp\ne0$, the relation $P'(t)=a^m$ of~\eqref{eq:root-derivative} reads
\begin{equation}\label{eq:homogeneous-derivative}
 (\partial_1P_h)(x,p)=p^{n-1}P'(t)=p^{n-1}a^m=p,
\end{equation}
since $a=1/p$ and $n-1-m=1$. Both sides are polynomials on the source,
so~\eqref{eq:homogeneous-derivative} holds everywhere.

\begin{lemma}[The cofactor is a unit]\label{lem:factor-unit}
Let $S$ belong to $G$. Then $H_h(x,p)\in\mathcal O(S)^\times$, and
\[
 p=H_h(x,p)\,(\partial_1G_h)(x,p)\qquad\text{on }S.
\]
\end{lemma}

\begin{proof}
Differentiating $P_h=G_hH_h$ and using $G_h(x,p)=0$ on $S$ turns
\eqref{eq:homogeneous-derivative} into the displayed formula.
If $H_h(x,p)$ vanished at a point of $S$, then $[x:p]$ would be a
common root of $G_h$ and $H_h$, hence a multiple root of $P_h$. This
contradicts Proposition~\ref{prop:base-change}.
\end{proof}

If $S$ has only constant units, for instance if $S$ is an affine
space, the cofactor is therefore a constant on $S$. This has strong
consequences.

\begin{lemma}[Constant cofactor]\label{lem:constant-cofactor}
Let $S$ belong to $G$, with $e\ge3$, and suppose $H_h(x,p)$ is a
constant $\theta$ on $S$. This holds whenever
$\mathcal O(S)^\times=\C^\times$, for instance when $S\simeq\A^k$. Then:
\begin{enumerate}
\item On $S\cap\{p\ne0\}$,
\[
 a^{e-2}=\theta\,G'(t),\qquad \theta\,a^{f}=H(t).
\]
\item On $I_G$, with $\kappa=\theta^{e-2+f}$,
\begin{equation}\label{eq:cofactor-identity}
 H_h^{\,e-2}X_0^f=\kappa\,(\partial_1G_h)^f,\qquad
 H_h^{\,e-2}X_1^f=(-1)^f\kappa\,(\partial_0G_h)^f .
\end{equation}
\item $H_h$ vanishes at no simple projective root of $G_h$. So $G$ and
$H$ share roots only at multiple roots of $G$.
\end{enumerate}
\end{lemma}

\begin{proof}
(1) Since $(\partial_1G_h)(x,p)=p^{e-1}G'(t)$,
Lemma~\ref{lem:factor-unit} gives $1=\theta p^{e-2}G'(t)$, which is the
first identity because $a=1/p$. Also $H_h(x,p)=p^fH(t)=\theta$, which
is the second.

(2) Eliminating $a$ gives $H(t)^{e-2}=\theta^{e-2}(\theta G'(t))^f$.
Homogenizing with $t=X_1/X_0$ gives the first identity
in~\eqref{eq:cofactor-identity} on $S$. Both sides are pulled back from
$I_G$, and $S\to I_G$, $s\mapsto(\gamma,[x:p])$, is dominant, so the
identity holds on the integral scheme $I_G$.

For the second identity, multiply the first by $X_1^f$. By Euler's
relation $X_1\partial_1G_h+X_0\partial_0G_h=eG_h=0$ on $I_G$, the
right-hand side becomes $\kappa(-X_0\partial_0G_h)^f$. Cancelling
$X_0^f$ on the dense open set $X_0\ne0$ gives the identity there, and it
extends to $I_G$.

(3) At a simple root $(\partial_1G_h,\partial_0G_h)\ne(0,0)$, and
Euler's relation excludes $\partial_1G_h=0$ when $X_0\ne0$, and
$\partial_0G_h=0$ when $X_1\ne0$. So one of the two identities
in~\eqref{eq:cofactor-identity} gives $H_h\ne0$, since $e-2\ge1$.
\end{proof}

\subsection{Balanced factorizations}

The constant-cofactor hypothesis holds automatically in an important
class of examples. Call the factorization \emph{balanced} if
\[
 H\equiv(G')^k\pmod G\qquad\text{and}\qquad f=k(e-2)
\]
for some $k\ge1$, so that $m=(k+1)(e-2)$. At a root $t$ of $G$ one has
$P'(t)=G'(t)H(t)=G'(t)^{k+1}$, so the marking equation over the
simple roots of $G$ factors as
\[
 a^m-G'(t)^{k+1}=\prod_{\zeta'^{\,k+1}=1}\bigl(a^{e-2}-\zeta'G'(t)\bigr).
\]
On a component $S$ belonging to $G$, the ratio $a^{e-2}/G'(t)$ is a
$(k+1)$-st root of unity, hence a constant $\zeta'$. Then
$H(t)=G'(t)^k=\zeta'^{-k}a^{k(e-2)}=\zeta'a^f$, so
$H_h(x,p)=p^fH(t)=\zeta'$. Thus every component belonging to $G$ has
constant cofactor $\theta=\zeta'$, whether or not its units are
constant, and conclusion (3) of Lemma~\ref{lem:constant-cofactor} holds
for all of them. For instance, for $n=5$, $e=3$ and $k=2$, the marking
is $a=\zeta'G'(t)$ on each component, a single-valued function of the
root.

\subsection{Factorized sectors}

Keep the assumptions of Lemma~\ref{lem:constant-cofactor}. Part (1) of
that lemma says that on $S$ the marking satisfies an equation involving
$G$ alone. We now show that $S$ is an open subset of the cover defined
by that equation.

Let $W\subset I_G$ be the open set of simple projective roots of $G_h$.
It is étale over $\Gamma$, hence smooth. Put $W_0=W\cap\{X_1=0\}$, the
locus where the root is $t=0$; this closed subscheme need not be
reduced, and nothing below requires it to be. Let $W^a\to W$ be the
$\mu_{e-2}$-torsor defined by
\[
 a^{e-2}=\theta G'(t)\quad\text{on the finite chart},\qquad
 v^{e-2}=-\theta\,Q_G'(u)\quad\text{on the infinity chart},
\]
glued by $v=ua$ as in~\eqref{eq:gluing}. It is étale over $W$ because
$G'(t)$, respectively $Q_G'(u)$, does not vanish at a simple root.

\begin{proposition}[Factorized sectors]\label{prop:sector}
Under the assumptions of Lemma~\ref{lem:constant-cofactor}, the map
$s\mapsto(\gamma,[x:p],a)$, with $a=1/p$ where $p\ne0$ and $v=1/x$
where $x\ne0$, is an open immersion
\[
 S\hookrightarrow W^a
\]
onto $S^a\setminus\{X_1=0,\ a\ne1\}$, for a unique connected component
$S^a$ of $W^a$. Consequently:
\begin{enumerate}
\item $S\setminus S_0\to W\setminus W_0$ is finite étale, of degree
$m_S=\deg(S^a/W)$, where $S_0=S\cap\{x=0\}$ is the \emph{zero branch}.
This degree divides $e-2$, and its Kummer deck group $\mu_{m_S}$ acts
on $S\setminus S_0$.
\item $S_0$ maps isomorphically onto the open and closed subset
$\{\omega\in W_0:(\omega,1)\in S^a\}$ of $W_0$. At these points
$\theta G'(0)=1$.
\item No point of $S$ lies over $I_G\setminus W$.
\end{enumerate}
\end{proposition}

\begin{proof}
\emph{The map lands in $W^a$.} A point of $S$ gives a simple root of
$P_h$, hence of $G_h$, so it lies over $W$. On $\{p\ne0\}$,
Lemma~\ref{lem:constant-cofactor} gives $a^{e-2}=\theta G'(t)$. For the
infinity chart, Euler's relation on $G_h(x,p)=0$ together with
$\partial_1G_h(x,p)=p/\theta$ gives
$p\bigl(\partial_0G_h(x,p)+x/\theta\bigr)=0$ on $S$. The function $p$
does not vanish identically on $S$: otherwise the root $[1:0]$ would
be a root of $G_h$ over a dense subset of $\Gamma$, so $X_0$ would
divide the irreducible form $G_h$ of degree $e\ge3$. Hence
$\partial_0G_h(x,p)=-x/\theta$ on $S$. With $u=p/x$ and homogeneity of
degree $e-1$, this reads $Q_G'(u)=\partial_0G_h(1,u)=-x^{-(e-2)}/\theta$,
which is the infinity equation $v^{e-2}=-\theta Q_G'(u)$ with $v=1/x$.

\emph{It is an open immersion.} The map is injective, because a point
of $Z_\Gamma$ is determined by $(\gamma,[t],a)$ or $(\gamma,u,v)$. It
is étale, since both $S$ and $W^a$ are étale over $\Gamma$, so its image
is open, and it is quasi-finite and birational onto that image. The
image is smooth, hence normal, and Zariski's Main Theorem
\cite{StacksZariskiMain} shows that the map is an open immersion. So
$S$ lies in a single connected component $S^a$, which is irreducible
because $W^a$ is smooth.

\emph{Its image.} On $S$ we have $\theta a^f=H(t)$ on the finite chart,
by Lemma~\ref{lem:constant-cofactor}, and $v^f=\theta^{-1}Q_H(u)$ on the
infinity chart, because $Q_H(u)=H_h(1,u)=x^{-f}H_h(x,p)=\theta v^f$. Since
$S$ is dense in the irreducible $S^a$, both identities hold on $S^a$.
So every point of $S^a$ in the finite chart satisfies
$a^{n-2}=a^{e-2}a^f=G'(t)H(t)=P'(t)$, and every point in the infinity
chart satisfies $v^{n-2}=-Q_G'(u)Q_H(u)=-Q'(u)$, using $G=0$ and
$Q_G=0$. In both charts the marking is nonzero, because $G'(t)$ and
$Q_G'(u)$ do not vanish on $W$, so the root is simple. By
Proposition~\ref{prop:base-change}, such a point lies in $U_\Gamma$
unless $t=0$ and $a\ne1$. The resulting set
$S^a\setminus\{X_1=0,\ a\ne1\}$ is irreducible, contained in
$U_\Gamma$, and contains $S$. Components of the smooth $U_\Gamma$ are
connected components, so this set is $S$.

\emph{The consequences.} (1) The group $\mu_{e-2}$ permutes the
connected components of $W^a$ transitively over each connected
component of $W$, so their common degree $m_S$ divides $e-2$, and the
stabilizer $\mu_{m_S}$ of $S^a$ acts on
$S^a\setminus\{X_1=0\}=S\setminus S_0$.
(2) Over $W_0$, the identity $\theta a^f=H(0)$ and the torsor equation
give $a^{n-2}=G'(0)H(0)=P'(0)=1$. So $a$ satisfies $a^{n-2}=1$ on the
part of $S^a$ over $W_0$, and is therefore locally constant there,
even if $W_0$ is not reduced. The locus $a=1$ is open and closed in it,
and, as $S^a\to W$ is finite étale, its image in $W_0$ is open and
closed. The map from it is injective, hence an isomorphism onto its
image, and this locus is $S_0$ by the description of the image. At
these points $1=a^{e-2}=\theta G'(0)$.
(3) The root of a point of $S$ is simple, because $U$ avoids
$E_{\mathrm{crit}}$.
\end{proof}

\begin{remark}
When $S$ is isomorphic to a plane, Proposition~\ref{prop:sector}
supplies the Kummer deck transformation of $S\setminus S_0$, its exact
fibre counts, and the zero branch $S_0$. The case $m_S=1$, a completely
split marking, makes $S$ birational to $W$; this always happens when
$e=3$, as in the balanced example above. For $e\le2$ there is nothing to
add. For $e=2$ a constant cofactor would give $\theta G'(t)=a^0=1$ at
both roots of the irreducible quadratic $G$, whereas $G'$ takes opposite
values at the two roots. For $e=1$ the marking is determined by
$a^{-1}=\theta G'(t)$, so $S$ has degree one over $\Gamma$.
\end{remark}

\section{Stable comparison and specialization}\label{sec:comparison}

Two polynomial maps are \emph{stably left--right equivalent} if, after
adjoining enough identity coordinates to each, polynomial automorphisms
of source and target identify them. Their generic field extensions are
then isomorphic after adjoining independent transcendental variables.
The normal closure and its permutation action on the source embeddings
are unchanged by such an extension. Generic permutation monodromy is
therefore a stable invariant.

The finite normalization together with its distinguished affine source
open is also intrinsic. Indeed, if $\mathcal B$ is the integral
closure of a target coordinate ring $\mathcal A$ in the source function
field $\mathcal L$, then $\mathcal B[s]$ is the integral closure of
$\mathcal A[s]$ in $\mathcal L(s)$. It is normal and integral over
$\mathcal A[s]$; conversely, an element integral over $\mathcal A[s]$
is integral over $\mathcal B[s]$ and hence belongs to it. This observation, also used in
the normalization framework of \cite{VanRijn2026NormalizationFunctoriality},
transports the boundary and its ramification and residue degrees under
stable equivalence.

\begin{proposition}\label{prop:weighted}
Let a degree-$N$ map to $\A^3_{A,B,C}$ have source function field
generated over $\C(A,B,C)$ by an element $w$ with minimal polynomial
\[
 \mathcal H(w)-BCw+cAC^2,
 \qquad \mathcal H\in\C[w],\quad \deg\mathcal H=N\ge2,\quad c\in\C^*.
\]
Its generic monodromy over $\C(A,B,C)$ is $S_N$.
\end{proposition}
\begin{proof}
On $C\ne0$, put $B'=BC$ and $A'=-cAC^2$. Together with $C$ these are
independent coordinates, so the polynomial becomes $\mathcal H(w)-B'w-A'$,
with the extra independent variable $C$. If $h_N$ is the leading
coefficient of $\mathcal H$, set $B'=-h_N\rho^{N-1}$ and $w=\rho z$, and
divide by $h_N\rho^N$. As $\rho$ tends to infinity, $\mathcal H(w)-B'w$
tends in these coordinates to $z^N+z$, which is Morse. So for some value
of $B'$ the polynomial $\mathcal H(w)-B'w$ is
Morse, and, as in Lemma~\ref{lem:morse}, the connected cover of the
$A'$-line has monodromy generated by transpositions, giving $S_N$. The
extra transcendental variable does not change that group.
\end{proof}

The antiderivative presentation and its full symmetric monodromy are
Theorems A--B of \cite{MikhailSzh2026WeightedLiftGalois}; see also
\cite{VanRijn2026MarkedRootFramework}. The proof is included to make
the comparison independent of a classification of other constructions.
For $n\ge4$, $G_n$ is imprimitive on $N=nm$ letters, whereas $S_N$
is primitive. Proposition~\ref{prop:weighted} and stable invariance
prove Corollary~\ref{cor:stable-main}. The argument distinguishes
single weighted lifts. It makes no assertion about their compositions,
or about all cancellation or quadratic-gauge maps.
The use of this stable invariant, including an explicit imprimitive
composition example, precedes this paper
\cite{VanRijn2026ImprimitiveFactorization}.

\subsection{An exact overlap with power-weighted lifts}\label{sec:power-overlap}
The power-weighted lifts of \cite{Annie2026PowerWeightedLifts} generalize
the ordinary weighted lifts of Proposition~\ref{prop:weighted}; when the
weight parameter is greater than one, their source field has an
additional cyclic stage, so Proposition~\ref{prop:weighted} does not
apply to them.
The degree-eight map $F_{2,3}=(F_1,F_2,F_3)$ of
\cite{Annie2026PowerWeightedLifts} can be written, with
$\alpha=x^2y$, as
\[
\begin{aligned}
F_1&=z(1+\alpha)^4+\frac{xy^2}{3}(5\alpha^3+17\alpha^2+20\alpha+8),\\
F_2&=4xz(1+\alpha)^3+\frac{y}{3}(20\alpha^3+48\alpha^2+33\alpha+2),\\
F_3&=-x^4z+\frac{x}{3}(3-5\alpha).
\end{aligned}
\]
Use the source coordinates $(x,u,w)=(x,s_3,s_4)$ of $E_{4,0}$ from
Section~\ref{sec:descent}, with outputs $(r,q,b)$.
Direct substitution gives
\begin{equation}\label{eq:power-overlap}
\begin{gathered}
 E_{4,0}\circ\Psi=\Lambda\circ F_{2,3},\\
 \Psi(x,y,z)=\left(x,y,-2z-4x^3y^3-\frac{20}{3}xy^2\right),
 \qquad \Lambda(A,B,C)=(-B,3A,C).
\end{gathered}
\end{equation}
Both are polynomial automorphisms: the third coordinate of $\Psi^{-1}$
is $-w/2-2x^3u^3-(10/3)xu^2$, and $\Lambda^{-1}(r,q,b)=(q/3,-r,b)$. Thus
the equivalence requires no stabilization. The determinants
$\det J\Psi=-2$ and $\det\Lambda=3$ are consistent with
$\det JE_{4,0}=1$ and $\det JF_{2,3}=-2/3$, since $1\cdot(-2)=3\cdot(-2/3)$.
The exact verification supplement checks every component of
\eqref{eq:power-overlap} against the recursive descent construction.

Consequently the degree-eight map $E_{4,0}$ is already present in the literature,
and the order-$192$ group computed here also describes that map.
No equivalence with higher power-weighted members or with the nonzero
parameter descents is asserted. In particular the ordinary-weighted
distinction cannot be strengthened to exclude all power-weighted lifts.

\subsection{Nonlinear sections}
The results above hold for every fixed coordinate parameter because
the polynomial charts and the Morse line were proved uniformly.
For a general nonlinear target subvariety $\Gamma\to Y$,
Proposition~\ref{prop:base-change} still describes the preimage
$U_\Gamma$ exactly, as an open subset of $Z\times_Y\Gamma$. But this
base change can be reducible or singular, normalization need not
commute with it, and a component of $U_\Gamma$ need not be an affine
space. Proposition~\ref{prop:sector} identifies such a component with
an open subset of a Kummer cover of the simple-root locus of one
factor; deciding whether it is an affine space requires further
geometric input.

The standard local cyclic model illustrates how degrees can drop. A
cover $a^m=s$ pulled back by $s=z^\ell$ becomes $a^m=z^\ell$. If
$g=\gcd(m,\ell)$, there are $g$ branches over $\C$, and the
normalization of each branch has degree $m/g$ over the $z$-line.
For $\ell=m$ this cyclic stage splits into degree-one branches.
This explains a possible drop in component degree, but does not
make the entire marked-root map injective or give a polynomial
affine-space chart for the new source.

\subsection{Further questions}
The generic group, boundary count, Picard rank and homotopy type are
constant across the fixed-coordinate parameters, so they cannot
distinguish the maps $E_{n,\lambda}$ from one another up to stable
equivalence. The actual maps
of the boundary divisors and the intersections of their images provide
finer candidate invariants. Other questions include monodromy on
nonlinear target graphs and the behaviour of their normalized boundary
under tangential specialization. These require arguments beyond the
coordinate-descent theorem.

\section{Exact verification and scope}\label{sec:verification}

The proofs establish the uniform statements. Symbolic computations
provide independent checks of formulas, finite examples and their
implementation; they are not substitutes for proofs in unbounded
dimension or for the topological gluing argument.

The accompanying scripts perform the following focused checks:
\begin{enumerate}
\item reconstruction on the finite, zero and infinity charts, together
with all multiplicity partitions in a finite range;
\item the first polynomial coordinate chart and its inverse, the
completion lemma, iterated descents through $n=8$, an all-parameter
$n=5$ calculation, and a transported collision;
\item the derivative-product identity, the integer Smith form of the
pair-valuation matrix, exact abstract permutation-group orders, and
Morse critical-value checks;
\item the two normalization charts, their overlap, the ramification
determinant, substitution of original source maps, and cyclic
specialization controls;
\item the polynomial identity with the published power-weighted
degree-eight map, its inverse source chart, determinant and collision;
\item the homogeneous derivative identity $(\partial_1P_h)(x,p)=p$ for
$n\le8$, the balanced-factor resultant identity, and numerical
reconstruction of every source point over random factorized targets,
together with $p=H_h(x,p)\,\partial_1G_h(x,p)$ at each of them
(Section~\ref{sec:sectors});
\item the identities derived in the text of Sections~\ref{sec:root}
and~\ref{sec:descent}: Lemma~\ref{lem:shape} for $n\le9$, the two
determinant factors in Proposition~\ref{prop:harish}, the $n=3$
example, the first chart and its boundary offsets for $n\le8$, the
explicit map $E_{4,0}$, the model case of Lemma~\ref{lem:completion},
and, for
Proposition~\ref{prop:induction}, the dependence of the offsets and the inverse-boundary
formula~\eqref{eq:induction-offset} for $n=5,6$, and the slopes $-\frac23,-\frac32,\frac23,\frac32$ for $n=8$.
\end{enumerate}
The scripts use only SymPy. The runner \texttt{run\_all.py} executes all
of them and writes a manifest with their hashes, exit statuses and run
times; a complete run takes about thirteen minutes.

\paragraph{Data availability.} The scripts, the runner and a reference
run form a self-contained supplement under the Apache License 2.0. Its
README maps each script to the statements it checks. The supplement is
distributed as the ancillary files of this paper on arXiv.

\section*{Acknowledgments}
AI assistants from Anthropic and OpenAI significantly contributed to writing
the manuscript, to the formal derivation and checking of its results, to the
verification scripts, and to revising the text.
\bibliographystyle{plain}
\bibliography{references}
\end{document}